\documentclass[10pt]{article}
\usepackage[hidelinks]{hyperref}
\usepackage[utf8]{inputenc}

\usepackage{geometry}
\usepackage[english]{babel}
\usepackage[backend=biber,style=alphabetic]{biblatex}
\bibliography{bibliography}
\DeclareFieldFormat{postnote}{#1}

\usepackage{mathtools}
\usepackage{amsmath}
\usepackage{amssymb}
\usepackage{graphicx}
\usepackage{booktabs}
\usepackage{array}
\usepackage{paralist}
\usepackage{verbatim}
\usepackage{subfig}
\usepackage{tikz}
\usepackage{tikz-cd}
\usetikzlibrary{arrows}
\tikzset{node distance=1.8cm, auto}
\usepackage{mathdots}

\usepackage{lmodern}
\usepackage{sectsty}
\usepackage{amsthm}
\theoremstyle{definition}
\usepackage{csquotes}
\allsectionsfont{\sffamily\mdseries\upshape}
\newtheorem{theorem}{Theorem}[section]
\newtheorem{proposition}[theorem]{Proposition}
\newtheorem{corollary}[theorem]{Corollary}
\newtheorem{definition}[theorem]{Definition}
\newtheorem{example}[theorem]{Example}
\newtheorem{lemma}[theorem]{Lemma}
\newtheorem{remark}[theorem]{Remark}

\newtheorem{algorithm}[theorem]{Algorithm}

\DeclareMathOperator\Tr{Tr}
\DeclareMathOperator\End{End}
\newcommand{\bF}{\mathbb{F}}

\title{\textsc{On Newton's identities in positive characteristic}}
\author{Sjoerd de Vries}
\date{}

\begin{document}
\maketitle

\begin{abstract}
    \noindent Newton's identities provide a way to express elementary symmetric polynomials in terms of power polynomials over fields of characteristic zero. In this article, we study the failure of this relation in positive characteristic and what can be recovered. In particular, we show how one can write the elementary symmetric polynomials as rational functions in the power polynomials over any commutative unital ring.
\end{abstract}


\tableofcontents

\section{Introduction}
The elementary symmetric polynomials, resp.~power polynomials, in $n$ variables are defined as 
\begin{align}\label{e_k}
e_k(x_1,\ldots,x_n) &= \sum_{1 \leq i_1 < \ldots < i_k \leq n} x_{i_1}x_{i_2}\cdots x_{i_k};\\  \label{p_k}
p_k(x_1,\ldots,x_n) &= \sum_{i=1}^n x_i^k.
\end{align}
In the following, we write $e_k$ (resp.~$p_k$) for the elementary symmetric (resp.~power) polynomials in $n$ variables, where $n$ is some fixed number which should be clear from the context. 
Equation~\eqref{e_k} holds for $k \geq 1$, and $e_0$ is defined to be~$1$. Equation~\eqref{p_k} holds for $k \geq 0$, so that $p_0(x_1,\ldots,x_n) = n$. In particular, $e_k = 0$ for $k > n$, but $p_k \neq 0$ for all~$k \geq 1$. For an integer partition~$\lambda$, which we will write in the form  $\lambda= (k_1,\ldots,k_l)$ with $k_1 \geq \ldots \geq k_l \geq 1$, we put  $e_\lambda := e_{k_1}\cdots e_{k_l}$ and $p_{\lambda} := p_{k_1}\cdots p_{k_l}$.\\

The importance of the elementary symmetric polynomials is made evident by the following classical result (cf.~Theorem~\ref{fundthmR}):
\begin{theorem}\label{fundthm}
Let $K$ be any field. Then the $K$-algebra of symmetric polynomials in $n$ variables is equal to $K[e_1,\ldots,e_n]$.
\end{theorem}

The power polynomials are related to the elementary symmetric polynomials via Newton's identities \cite[p.23]{mcdon}:
\begin{equation}\label{newt_eq}
ke_k= \sum_{i=1}^k (-1)^{i-1}e_{k-i}p_i.
\end{equation}

An immediate corollary of Theorem~\ref{fundthm} and Newton's identities is the following:

\begin{theorem}\label{classical}
    Let $K$ be a field of characteristic zero. Then $K[e_1,\ldots,e_n] = K[p_1,\ldots,p_n]$.
\end{theorem}

On the other hand, suppose $K$ has characteristic $r>0$ (we avoid using~$p$ to avoid confusion with the power polynomials). Then we have algebraic relations
    \begin{equation}\label{freshman}
    p_{kr} = x_1^{kr} + \ldots + x_n^{kr} = (x_1^k + \ldots + x_n^k)^r = p_k^r
    \end{equation}
    for any integer $k \geq 1$. Thus, contrary to Theorem~\ref{classical}, it is not true that $K[e_1,\ldots,e_n] = K[p_1,\ldots,p_n]$ when $n\geq r$, because $K[e_1,\ldots,e_n]$ has transcendence degree~$n$ while $K[p_1,\ldots,p_n]$ has transcendence degree at most $n-\lfloor n/r \rfloor$.
    However, one may wonder if the set of \emph{all} power polynomials still generates the ring of symmetric polynomials over fields of positive characteristic.\\

This question has a negative answer, as the following example shows. Suppose we are working with symmetric polynomials in two variables over~$\mathbb F_2$. Then $p_k(1,1) = 1^k+1^k = 0$ for all $k \geq 0$. If $e_2 = x_1x_2$ could be written as an algebraic combination of $\{p_k\}_{k \in \mathbb N}$, i.e.~as a finite sum over integer partitions
\begin{equation}\label{false}
e_2 = c_0 + \sum_{\lambda}c_\lambda p_\lambda
\end{equation}
with $c_0, c_{\lambda} \in \mathbb F_2$, then evaluating both sides at $(1,1)$ gives $c_0 = 1$, whereas evaluating at $(0,0)$ gives $c_0 = 0$. Hence no such relation~\eqref{false} can exist.\\

However, notice that in the above case, one can write $e_2$ in two variables in terms of power polynomials as
\begin{equation}\label{mod2ex}
    e_2 = x_1x_2 = \frac{x_1^2 x_2 + x_1x_2^2}{x_1+x_2} = \frac{(x_1^3 + x_1^2 x_2 + x_1x_2^2 + x_2^3) + x_1^3 + x_2^3}{x_1+x_2}=\frac{p_1^3+p_3}{p_1}.
\end{equation}
A natural question is then whether it is true that the elementary symmetric polynomials can always be obtained as rational functions in the power polynomials. This question has a positive answer, as shown by Sch\"onhage \cite[Theorem 2]{schonhage}:

\begin{theorem}\label{mainthm0}
    Let $K$ be any field. Then $K(e_1,\ldots,e_n) = K(p_1,p_2,\ldots)$.
\end{theorem}

In this paper, we give a new proof which holds over any base ring:

\begin{theorem}[Theorem~\ref{mainthm}]\label{mainthmintro}
    Let $R$ be a commutative ring with unity. Then \[
    Q(R[e_1,\ldots,e_n]) = Q(R[p_1,p_2,\ldots]),
    \]
    where $Q(-)$ denotes the total ring of fractions.
\end{theorem}

The proof is based on a simple algorithm to express elementary symmetric polynomials as rational functions of power polynomials.\\

We also study several related questions. In \S\ref{sec3}, we study
the subalgebra generated by the power polynomials over fields of positive characteristic and show that it is not finitely generated. In \S\ref{sec2}, we prove Theorem~\ref{mainthmintro} and discuss applications. In \S\ref{sec:dvr}, we consider the case of a discretely valued field, where we relate the valuations of power sums to the valuations of the individual elements.\\

Although symmetric polynomials are central objects in algebra and representation theory, it seems that their theory in positive characteristic has thus far been somewhat neglected. There is literature dealing with the problem of recovering the roots of a polynomial from its power sums in positive characteristic, as this leads to fast algorithms for several computational problems, see~\cite{bostan-FSS}. The focus in this literature has been on complexity theory rather than algebraic results.

\section{The subalgebra generated by power polynomials}\label{sec3}

Let $R$ be any commutative ring with unity. We first note that Theorem~\ref{fundthm} continues to hold:

\begin{theorem}[Fundamental Theorem on Symmetric Polynomials]\label{fundthmR}
    Let $R$ be a commutative ring with unity. Then $R[e_1,\ldots,e_n] \subset R[x_1,\ldots,x_n]$ is equal to the ring of symmetric polynomials over $R$, and the set $\{e_1,\ldots,e_n\}$ is algebraically independent.
\end{theorem}

\begin{proof}
    Let $S_n$ denote the symmetric group on $n$ letters, which acts on the variables $x_1,\ldots,x_n$ by permutation. By \cite[I.2.4]{mcdon}, the theorem holds with $R = \mathbb Z$. In particular, if we let $S_n$ act trivially on $R$, we have
    \[
    R[e_1,\ldots,e_n] = R \otimes \mathbb Z[e_1,\ldots,e_n] = R \otimes \mathbb Z[x_1,\ldots,x_n]^{S_n} = R[x_1,\ldots,x_n]^{S_n}.
    \]
    This implies the first part of the theorem. The algebraic independence follows from \cite[I.2.3]{mcdon}, which is stated for~$\mathbb Z$ but whose proof works over any ring~$R$.
\end{proof}

\begin{remark}
    The fact that $\{e_1,\ldots,e_n\}$ is algebraically independent over any ring $R$ does not follow directly from the statement for $R = \mathbb Z$ without using properties of the elementary symmetric polynomials. For instance, as shown in \cite[Example 1.4]{transdeg}, there are rings $R$ and polynomials $f_1,\ldots,f_m \in \mathbb Z[x_1,\ldots,x_n]$ such that 
    \[
    \text{trdeg}_R\{f_1,\ldots,f_m\} < \text{trdeg}_{\mathbb Z}\{f_1,\ldots,f_m\},
    \]
    even when none of the coefficients of the $f_i$ are zero divisors in $R$. One example of this phenomenon is that for a prime number $r$,
    \[
    \text{trdeg}_{\mathbb F_r}\{p_1,p_r\} = 1
    \]
    for any number of variables $n$.
\end{remark}

By applying Newton's identities inductively, one obtains the following expression for the power polynomials in terms of the elementary symmetric polynomials:
    \begin{equation}\label{totalnewton}
    p_m = (-1)^m \sum_{t_1 + 2t_2 + \ldots + mt_m = m} c_{t_1,\ldots,t_m}\prod_{i=1}^m (-e_i)^{t_i},
    \end{equation}
where
\begin{equation*}
    c_{t_1,\ldots,t_m}=\frac{m \cdot (t_1+t_2 + \ldots + t_m - 1)!}{t_1!t_2!\cdots t_m!}.
\end{equation*}   
    Since the coefficients $c_{t_1,\ldots,t_m}$ are integers, this is a formula that holds over~$R$. However, it is not clear from the formula~\eqref{totalnewton} how to judge whether a given symmetric polynomial is generated by power polynomials.
    
    \begin{definition}
        Let $R$ be a commutative ring with unity. We denote by
        \[
        R[p_\infty] := \bigcup_{k \geq 1} R[p_1,\ldots,p_k] \subseteq R[e_1,\ldots,e_n]
        \]
        the subalgebra generated by power polynomials.
    \end{definition}
    
A natural question is thus whether one can describe the subalgebra $R[p_\infty] \subseteq R[e_1,\ldots,e_n]$ in terms of elementary symmetric polynomials.
In this section, we consider this question when $R = K$ is a field of characteristic $r > 0$. We work in the $K$-basis $\{ e_\lambda \ | \ \lambda \text{ is an integer partition}\}$ of $K[e_1,\ldots,e_n]$, and when we write ``monomial", we mean a monomial expressed in terms of elementary symmetric polynomials; that is, a constant multiple of some~$e_\lambda$. We assume that $n \geq r$, as otherwise the situation is identical to the characteristic~0 case.

\begin{proposition}\label{improvement}
 Denote by $E$ the $K$-algebra generated by all monomials $e_\lambda$ such that at least one part of $\lambda$ is coprime to $r$. Then
 \[
 K[p_\infty] \subseteq E \subsetneq K[e_1,\ldots,e_n].
 \]
\end{proposition}

\begin{proof}
Since $p_{kr} = p_k^r$ for all~$k$, we see that $K[p_\infty]$ is generated by power polynomials $p_k$ with $(k,r) = 1$. Any such power polynomial lies in~$E$. Indeed, any monomial $e_\lambda$ corresponding to a partition $\lambda$ with all parts divisible by~$r$ has degree a multiple of~$r$, while $p_k$ is homogeneous of degree~$k$, which by assumption is not divisible by~$r$.
The first inclusion now follows because $E$ is a subalgebra of $K[e_1,\ldots,e_n]$. To see that $E$ is a proper subalgebra, note that e.g.~$e_r \notin E$.
\end{proof}

\begin{example}
    One can check by hand that for $K = \mathbb F_2$ and $n = 3$, one cannot write $e_3 = x_1x_2x_3$ as an algebraic combination of power polynomials. More generally, one can show that the containment $K[p_\infty] \subseteq E$ is an equality if and only if $n = r$.
\end{example}

\begin{proposition}\label{chain}
    Let $K$ be a field of characteristic $r > 0$. Fix an integer~$k$ not divisible by~$r$.
    Then 
    \[
    K[p_1,\ldots,p_{k-1}] \subsetneq K[p_1,\ldots,p_{k}].
    \]
\end{proposition}

\begin{proof}
Write $k = ar + b$ for integers $a \geq 0$ and $0 < b < r$. Applying the identity~\eqref{totalnewton} with $m = k$, $t_b = 1$, and $t_r = a$ tells us that $p_k$ contains the monomial $e_r^ae_b$ with non-zero coefficient.\\

    We claim that there are no elements in $K[p_1,\ldots,p_{k-1}]$ which, when written in terms of the elementary symmetric polynomials, contain the monomial $e_r^ae_b$; this will finish the proof. Suppose for contradiction that $f$ is a counterexample. Consider an expression
    \begin{equation*}
    f = \sum_{v \in \mathbb N^{k-1}} c_vp_1^{v_1}\ldots p_{k-1}^{v_{k-1}},
    \end{equation*}
    with $v_i = 0$ for all $r \mid i$; we can achieve this because $p_{jr} = p_j^r$. By the algebraic independence of the elementary symmetric polynomials, the monomial $e_r^ae_b$ in $f$ must be a product of monomials of the form $e_r^\alpha e_b^\beta$, with $\alpha \leq a$ and $\beta \leq 1$, occurring in $p_1^{v_1},\ldots,p_{k-1}^{v-1}$. Since there are no monomials of this form with $\beta = 0$ by Proposition~\ref{improvement}, we see that the entire monomial $e_r^ae_b$ must occur in some~$p_m$. This is a contradiction for degree reasons: we obtain $m = ra + b = k > m$. This finishes the proof.
\end{proof}

\begin{corollary}
    If $n \geq r$, the subalgebra $K[p_\infty] \subset K[e_1,\ldots,e_n]$ is not finitely generated.
\end{corollary}

\begin{proof}
    Given any finite set $S = \{f_1,\ldots,f_m\}$ of algebraic combinations of power polynomials in $n$ variables, we have $K[S] \subseteq K[p_1,\ldots,p_N]$ for $N = \max \{\deg f_i\}$. Thus Proposition~\ref{chain} implies that $K[S]$ is properly contained in $K[p_\infty]$.
\end{proof}

\begin{example}
    Let $K = \mathbb F_2$ and $n = 2$. Then we have for any $N = 2m+1$:
    \[
    \mathbb F_2[p_1,p_3,\ldots,p_N] = \mathbb F_2[e_1,e_1e_2,e_1e_2^2,\ldots,e_1e_2^m].
    \]
    In particular, $K[e_1,\ldots,e_n]$ is not the integral closure of $K[p_\infty]$.
\end{example}

\section{Elementary symmetric polynomials in terms of power polynomials}\label{sec2}

Let $R$ be a commutative ring with unity. In this section, we give an effective proof that any symmetric polynomial over $R$ can be expressed as a rational function in the power polynomials. We then consider the case where $R$ contains a field of positive characteristic, where some improvements can be made.

\subsection{The case of a general ring}\label{subsec1}
Let $R$ be any commutative ring with unity. For $k \geq 0$, we view $e_k$, $p_k$ as elements of $R[x_1,\ldots,x_n]$. We first prove some preliminary results needed in the proof of Theorem~\ref{mainthm}.\\

We fix the following notation. For a polynomial $f \in R[x_1,\ldots,x_n]$ and $k \in \mathbb N$, let
\[
f(\hat{x}_k) := f(x_1,\ldots,\hat{x}_k,\ldots,x_n) := f(x_1,\ldots,0,\ldots,x_n).
\]
In particular, $f(\hat{x}_k) = f$ if $k > n$. 

\begin{lemma}\label{lemma1}
    Let $M = (m_{i,j})$ be a $d \times d$ matrix with entries given by power polynomials: $m_{i,j} = p_{r_{i,j}}(x_1,\ldots,x_{n_{i,j}})$ for some integers $r_{i,j}, n_{i,j} \geq 1$.
    Define 
    \begin{equation*}
    \epsilon_{i,j,k} := \begin{cases} x_k^{r_{i,j}} & k \leq n_{i,j}; \\ 0 & \text{otherwise}.\end{cases}
    \end{equation*}
    Then for any $1 \leq i \leq d$ and any $k \geq 1$, we have
    \begin{equation*}
    \det M = \det \begin{pmatrix}
        m_{1,1} & m_{1,2} & \cdots & m_{1,d} \\
        \vdots & \vdots & \ & \vdots \\
        m_{{i-1},1} & m_{{i-1},2} & \cdots & m_{{i-1},d} \\
        \epsilon_{i,1,k} & \epsilon_{i,2,k} & \cdots & \epsilon_{i,d,k}\\
        m_{{i+1},1} & m_{{i+1},2} & \cdots & m_{{i+1},d} \\
        \vdots & \vdots & \ & \vdots \\
        m_{d,1} & m_{d,2} & \cdots & m_{d,d}
    \end{pmatrix}
    + \det \begin{pmatrix}
    m_{1,1} & m_{1,2} & \cdots & m_{1,d} \\
        \vdots & \vdots & \ & \vdots \\
        m_{{i-1},1} & m_{{i-1},2} & \cdots & m_{{i-1},d}\\
        m_{i,1}(\hat{x}_k) & m_{i,2}(\hat{x}_k) & \cdots & m_{i,d}(\hat{x}_k) \\
        m_{{i+1},1} & m_{{i+1},2} & \cdots & m_{{i+1},d} \\
        \vdots & \vdots & \ & \vdots \\
        m_{d,1} & m_{d,2} & \cdots & m_{d,d}
    \end{pmatrix}.
    \end{equation*}
\end{lemma}
\begin{proof}
    This follows from the $i$-th row expansion of $\det M$, since $m_{i,j} = \epsilon_{i,j,k} + m_{i,j}(\hat{x}_k)$.
\end{proof}

\begin{proposition}\label{prop1}
    Let $d,n \in \mathbb Z_{\geq 1}$ and consider the Hankel matrix
    \begin{equation}\label{P_n} P_{d,n}(x_1,\ldots,x_n) =
    \begin{pmatrix}
        p_1 & p_2 & p_3 & \cdots & p_d \\
        p_2 & p_3 & p_4 & \cdots & p_{d+1} \\
        p_3 & p_4 & p_5 & \cdots & p_{d+2}\\
        \vdots & \vdots & \vdots & \iddots & \vdots \\
        p_d & p_{d+1} & p_{d+2} & \cdots & p_{2d-1}
    \end{pmatrix}.
    \end{equation}
    Then the following hold:
    \begin{enumerate}
    \item For any $d$, we have
    \begin{equation*}
        \det P_{d,d} = e_d \prod_{1 \leq i < j \leq d}(x_i - x_j)^2.
    \end{equation*}
    \item For any $d$ and $n$, we have
    \begin{equation*}
        \det P_{d,n} = \sum_{1 \leq i_1 < i_2 < \ldots < i_d \leq n}\det P_{d,d}(x_{i_1},x_{i_2},\ldots,x_{i_d}).
    \end{equation*}
    \end{enumerate}
    In particular, $\det P_{d,n} \in R[x_1,\ldots,x_n]$ is non-zero if and only if $d \leq n$.
\end{proposition}

\begin{proof}  
    \textbf{Proof of 1.} If we write
    \begin{equation*}
    P_{d,d} =
    \begin{pmatrix}
        1 & 1 & \cdots & 1 \\
        x_1 & x_2 & \cdots & x_d \\
        x_1^2 & x_2^2 & \cdots & x_d^2 \\
        \vdots & \vdots & \ddots & \vdots \\
        x_1^{d-1} & x_2^{d-1} & \cdots & x_d^{d-1}
    \end{pmatrix}
    \begin{pmatrix}
        x_1 & x_1^2 & \cdots & x_1^d \\
        x_2 & x_2^2 & \cdots & x_2^d \\
        x_3 & x_3^2 & \cdots & x_3^d \\
        \vdots & \vdots & \ddots & \vdots \\
        x_d & x_d^2 & \cdots & x_d^d
    \end{pmatrix},
    \end{equation*}
    then the first matrix is a standard Vandermonde matrix, which has determinant 
    \begin{equation*}\prod_{1 \leq i < j \leq d}(x_j-x_i).
    \end{equation*}
    The second matrix is of the same form after dividing row~$i$ by~$x_i$ and taking the transpose; hence it has determinant $e_d \prod_{1 \leq i < j \leq d}(x_j-x_i)$. The result follows.\\

    \textbf{Proof of 2.} Fix $d$. The idea is to decompose $\det P_{d,n}$ completely using Lemma~\ref{lemma1}. First consider the case $n=d$. By applying Lemma~\ref{lemma1} with $i = k=1$, we obtain
    \begin{equation*}
    \det P_{d,d} = \det \begin{pmatrix}
        x_1 & x_1^2 & \cdots & x_1^d \\
        p_2 & p_3 & \cdots & p_{d+1} \\
        \vdots & \vdots & \ddots & \vdots \\
         p_d & p_{d+1} & \cdots & p_{2d-1}
    \end{pmatrix}
    +\det \begin{pmatrix}
        p_1(\hat{x}_1) & p_2(\hat{x}_1) & \cdots & p_d(\hat{x}_1) \\
        p_2 & p_3 & \cdots & p_{d+1}\\
        \vdots & \vdots & \ddots & \vdots \\
        p_d & p_{d+1} & \cdots & p_{2d-1}
    \end{pmatrix}.
    \end{equation*}
    In the first matrix, one can replace each power polynomial $p_j$ with $p_j(\hat{x}_1)$, by replacing row~$i$ with row $i - (x_1^{i-1}$ times row 1) for all $i = 2,\ldots,d$. The second determinant we can reduce again by applying Lemma~\ref{lemma1} with $i=1$ and $k=2$. Continuing this process gives the expression
    \begin{equation*}
        \det P_{d,d} = \sum_{k=1}^d \det \begin{pmatrix}
        x_k & x_k^2 & \cdots & x_k^d \\
        p_2(\hat{x}_k) & p_3(\hat{x}_k) & \cdots & p_{d+1}(\hat{x}_k) \\
        \vdots & \vdots & \ddots & \vdots \\
         p_d(\hat{x}_k) & p_{d+1}(\hat{x}_k) & \cdots & p_{2d-1}(\hat{x}_k)
    \end{pmatrix}.
    \end{equation*}
    By applying Lemma~\ref{lemma1} inductively to the other rows, one obtains the following expression for~$\det P_{d,d}$:
    \begin{equation}\label{pdd}
        \det P_{d,d} = \sum_{\sigma \in S_d} \det \begin{pmatrix}
            x_{\sigma(1)} & x_{\sigma(1)}^2 & \cdots & x_{\sigma(1)}^d \\
            x_{\sigma(2)}^2 & x_{\sigma(2)}^3 & \cdots & x_{\sigma(2)}^{d+1} \\
            \vdots & \vdots & \ddots & \vdots \\
            x_{\sigma(d)}^{d} & x_{\sigma(d)}^{d+1} & \cdots & x_{\sigma(d)}^{2d-1}
        \end{pmatrix}
    \end{equation}
    where $S_d$ denotes the symmetric group on $d$ letters.\\
    
    For arbitrary $n$, one can use the same method to express $\det P_{d,n}$ as
    \begin{equation*}
    \det P_{d,n} = \sum_{1 \leq i_1,\ldots,i_d \leq n} \det \begin{pmatrix}
        x_{i_1} & x_{i_1}^2 & \cdots & x_{i_1}^d \\
            x_{i_2}^2 & x_{i_2}^3 & \cdots & x_{i_2}^{d+1} \\
            \vdots & \vdots & \ddots & \vdots \\
            x_{i_d}^d & x_{i_d}^{d+1} & \cdots & x_{i_d}^{2d-1}
    \end{pmatrix}
    \end{equation*}
    where such a determinant is automatically zero if there is a repetition in the indices (since one row will be a multiple of another). By Equation~\eqref{pdd}, we may write this determinant as
    \begin{equation*}
    \det P_{d,n} = \sum_{1 \leq i_1 < i_2 < \ldots < i_d \leq n}\det P_{d,d}(x_{i_1},x_{i_2},\ldots,x_{i_d}).
    \end{equation*}
    When $d \leq n$, the above expression is non-zero because it contains the monomial $x_1x_2^{3}\ldots x_d^{2d-1}$. This finishes the proof.
\end{proof}

\begin{remark}
    The method used in the proof of part~2 of Proposition~\ref{prop1} works for any matrix $M = (m_{i,j})$ consisting of power polynomials in $n$ variables such that $\deg m_{i,j} - \deg m_{i+1,j}$ is independent of~$j$ for all~$i$. In other words, given such a matrix $M$, one can apply Lemma~\ref{lemma1} repeatedly to obtain an expression for $\det M$ in terms of a simpler determinant. 
\end{remark}

In order to state the main theorem over general rings $R$, we fix the following notation.

\begin{definition}
    Let $R$ be a commutative ring with unity and let $\{f_i \hspace{0.3em} | \hspace{0.3em} i \in I\}$ be a set of polynomials in $R[x_1,\ldots,x_n]$. Let $S$ denote the multiplicative system of non-zero divisors in $R[f_i \hspace{0.3em} | \hspace{0.3em} i \in I] \subset R[x_1,\ldots,x_n]$. We denote by 
    \begin{equation*}
        R(f_i \hspace{0.3em} | \hspace{0.3em} i \in I) := S^{-1}R[f_i \hspace{0.3em} | \hspace{0.3em} i \in I] = Q(R[f_i \hspace{0.3em} | \hspace{0.3em} i \in I])
    \end{equation*} the total ring of fractions of $R[f_i \hspace{0.3em} | \hspace{0.3em} i \in I]$.
\end{definition}

\begin{lemma}[{\cite[Exercise 1.2]{atimcdon}}] \label{lemma2}
    Let $R$ be a commutative ring with unity, and let $f$ be a zero divisor in $R[x_1,\ldots,x_n]$. Then there exists a zero divisor $r \in R$ such that $rm=0$ for any monomial $m$ in~$f$.
\end{lemma}

\begin{theorem}\label{mainthm}
    Let $R$ be any commutative unital ring not containing $\mathbb Q$, and let $r_0 \geq 2$ be the smallest integer which is not invertible in~$R$. Fix the number of variables $n \geq r_0$. Then 
    \[
    R(e_1,\ldots,e_n) = R(p_1,\ldots,p_{2n+1-r_0}).
    \]
\end{theorem}

\begin{proof}
The power polynomials are symmetric polynomials, so $p_k \in R(e_1,\ldots,e_n)$ for all~$k$ by Theorem~\ref{fundthmR}. We prove the converse direction of the theorem by providing an algorithm to compute $e_1,\ldots,e_n$ as rational functions in the power polynomials, and then prove that this algorithm always works.

\begin{algorithm}\hfill \label{algR} \\
\textsc{Input:} Power polynomials $p_1,p_2,\ldots,p_{2n+1-r_0}$.\\
\textsc{Output:} Expressions for $e_1,e_2,\ldots,e_n$ as rational functions of power polynomials.
    \begin{enumerate}
        \item For $k < r_0$, apply the usual Newton identities
\begin{equation*}
ke_k= \sum_{i=1}^k (-1)^{i-1}e_{k-i}p_i
\end{equation*}
recursively to obtain expressions for $e_k$ in terms of power polynomials.
\item Solve the system of equations
\begin{equation}\label{matrixeq}
\begin{pmatrix}
p_1 & -p_2 & \cdots & (-1)^{n} p_{n+1} \\
-p_2 & p_3 & \cdots & (-1)^{n+1}p_{n+2} \\
\vdots & \vdots & \ddots & \vdots \\
(-1)^{n-r_0}p_{n+1-r_0} & \cdots & \cdots & (-1)^{k}p_{2n+1-r_0}
\end{pmatrix}
\begin{pmatrix}
    e_{n} \\ e_{n-1} \\ \vdots \\ e_1 \\ 1
\end{pmatrix}
=
\begin{pmatrix}
    0 \\ 0 \\ \vdots \\ 0 \\ 0
\end{pmatrix}
\end{equation}
by reducing it to the form 
\begin{equation}\label{reducedsystem}
\begin{pmatrix}
1 & 0 & \cdots & 0 & c_{1,n+2-r_0} & \cdots & c_{1,n+1} \\
0 & 1 & \cdots & 0 & c_{2,n+2-r_0} & \cdots & c_{2,n+1} \\
\vdots & \vdots & \ddots & \vdots & \vdots & \ddots & \vdots\\
0 & 0 & \cdots & 1 & c_{n+1-r_0,n+2-r_0} & \cdots & c_{n+1-r_0,n+1}
\end{pmatrix}
\begin{pmatrix}
    e_n \\ e_{n-1} \\ \vdots \\ e_1 \\ 1
\end{pmatrix}
=
\begin{pmatrix}
    0 \\ 0 \\ \vdots \\ 0 \\ 0
\end{pmatrix}
\end{equation}
by performing elementary row operations over $R(p_1,\ldots,p_{2n+1-r_0})$.
\item For $r_0 \leq k \leq n$, this gives the expressions
\[
e_k = -\sum_{i=0}^{r_0 - 1}c_{n+1-k,n+1-i}e_i
\]
with each $c_{a,b} \in R(p_1,\ldots,p_{2n+1-r_0})$.
    \end{enumerate}
\end{algorithm}

Step~2 of the algorithm needs to be justified. Since $e_N=0$ for $N>n$, the Newton identities imply that
\begin{equation}\label{newt_eq1}
\sum_{i=N-n}^N (-1)^{i-1} e_{N-i}p_i = 0
\end{equation}
for all $N>n$. Applying \eqref{newt_eq1} with $N = n+1,\ldots,2n+1-r_0$ yields the system of equations~\eqref{matrixeq}.\\

The $(n+1-r_0) \times (n+1)$ matrix in this equation is of the form $(\tilde{P}|*)$ for a square matrix $\tilde{P}$. Multiplying each second row of $\tilde{P}$ by $-1$ and then multiplying each second column by $-1$, we see that 
\begin{equation*}
\det \tilde{P} = \det P_{n+1-r_0,n},
\end{equation*}
with notation from~\eqref{P_n}. This determinant is not a zero-divisor in $R[x_1,\ldots,x_n]$ by Proposition~\ref{prop1} and Lemma~\ref{lemma2}, since $\det P_{d,n}$ for $d \leq n$ always contains the monomial $x_1x_2^3\ldots x_d^{2d-1}$ with coefficient $1 \in R^\times$. Hence the matrix in~\eqref{matrixeq} can be reduced to the form $(I_{n+1-r_0}|*)$ by performing elementary row operations over $R(p_1,\ldots,p_{2n+1-r_0})$. This finishes the proof.
\end{proof}

\begin{remark}\label{remk}
    The reduction of the system \eqref{matrixeq} to the system \eqref{reducedsystem} only requires division by $\det P_{n+1-r_0,n}$, which we know explicitly by Proposition~\ref{prop1}. Thus, one only needs to invert algebraic expressions in the power polynomials $p_i$ for $1 \leq i \leq 2n+1-2r_0$ (not $2n+1-r_0$) to obtain all elementary symmetric polynomials.
Moreover, this observation gives a test as to whether one can compute $e_k(\alpha_1,\ldots,\alpha_n)$ for $k \notin R^\times$ given only the $p_i(\alpha_1,\ldots,\alpha_n)$ for $i \leq 2n+1-k$ and the $e_i(\alpha_1,\ldots,\alpha_n)$ for $i < k$: namely, it suffices that 
    \[
    \det P_{n+1-k,n}(\alpha_1,\ldots,\alpha_n) \in R^\times.
    \]
Indeed, this is the denominator in the expression for $e_k(\alpha_1,\ldots,\alpha_n)$ obtained from Algorithm~\ref{algR}. Moreover, since $e_k$ is a polynomial, the denominator divides the numerator whenever the denominator is non-zero; hence if~$R$ is an integral domain, the denominator need only be non-zero.

For further discussion about the evaluation of the rational expressions, see~\S\ref{subsec:eval}.
\end{remark}

\begin{remark}
    In characteristic~0, the expressions for the elementary symmetric polynomials in terms of power polynomials obtained from Newton's identities are valid for all~$n$; in other words, they are identities of symmetric functions. In contrast, the expressions obtained from Algorithm~\ref{algR} do depend on~$n$.
\end{remark}

\begin{example}
Implementing the algorithm from the proof, we obtain for $r_0=2$:
\begin{align*}
e_2(x_1,x_2) &= \frac{p_1p_2 - p_3}{p_1};\\
e_2(x_1,x_2,x_3) &= \frac{p_1p_2p_3 - p_1^2p_4 - p_2p_4 + p_1p_5}{p_2^2 - p_1p_3};\\
    e_2(x_1,x_2,x_3,x_4) &= \frac{p_1p_3^2p_4 + p_1^2p_4p_5 + p_1p_2^2p_6 + p_2p_4p_5 + p_2p_3p_6 + p_1p_3p_7}{p_3^3 - p_1p_3p_5 - 2p_2p_3p_4 + p_1p_4^2 + p_2^2p_5 } \\
    &\qquad - \frac{p_1p_2p_4^2 + p_1p_2p_3p_5 + p_3^2p_5 + p_1^2p_3p_6 + p_1p_4p_6 + p_2^2p_7}{p_3^3 - p_1p_3p_5 - 2p_2p_3p_4 + p_1p_4^2 + p_2^2p_5 }.
\end{align*}
Note that the expression for $e_2(x_1,x_2)$ reduces to Equation~\eqref{mod2ex} modulo~2, since $p_2 = p_1^2$ in characteristic~2. For $r_0 \in \{2,3\}$, we obtain
\begin{align}
\label{eqmod3}
\begin{split}
    e_3(x_1,x_2,x_3) &= \frac{-p_1p_3 + p_2e_2 + p_4}{p_1};\\
    e_3(x_1,x_2,x_3,x_4) &= \frac{p_1^2p_5 -p_1p_2p_4 - p_1e_2p_4 - p_1p_6 + p_2e_2p_3 + p_2p_5}{p_2^2 - p_1p_3}.
    \end{split}
\end{align}
\end{example}

\begin{remark}
    The complete homogeneous symmetric polynomials are defined by
    \[
    h_k(x_1,\ldots,x_n) = \sum_{1 \leq i_1 \leq i_2 \leq \ldots 
 \leq i_k \leq n} x_{i_1}x_{i_2}\ldots x_{i_k}.
    \]
    One can express the power polynomials in terms of the complete homogeneous symmetric polynomials without problems, but going the other way, there are again denominators. Fortunately, the fact that $\mathbb Z[e_1,e_2,\ldots,e_n] = \mathbb Z[h_1,h_2,\ldots,h_n]$ \cite[I.2.8]{mcdon} in conjunction with Theorem~\ref{mainthm} allows one to compute the~$h_k$ in terms of the power polynomials over any ring.
\end{remark}

\subsection{The case of \texorpdfstring{$\mathbb F_r$}{Fr}-algebras}\label{subsec2}

Fix a prime number $r$, and suppose from now on that $R$ is a commutative $\mathbb F_r$-algebra. We describe an algorithm to compute $e_1,\ldots,e_n$ as rational functions in the power polynomials when $n \geq r$. The algorithm is similar to Algorithm~\ref{algR}, but takes fewer power polynomials as input.\\

The fact that this is possible is due to Sch\"onhage \cite{schonhage}, who considered the case where $R$ is a field of positive characteristic~$r$. We give a new method to achieve Sch\"onhage's result in the spirit of Algorithm~\ref{algR}. It may be useful to compare this section with the previous one to see what improvements can be made when one has more information about the ring $R$.\\

Let $\ell := n + \lfloor (n-1)/(r-1)\rfloor$. By \cite[Theorem 2]{schonhage}, the elementary symmetric polynomials are elements of $R(p_1,p_2,\ldots,p_\ell)$. Note that this is the best possible result: since $p_{rk} = p_k^r$ for each $k \geq 1$, the transcendence degree of $\{p_1,p_2,\ldots,p_m\}$ is at most $m - \lfloor m/r \rfloor$, and $\ell$ is the least integer such that this quantity equals $n$. Since $\ell \leq 2n+1-r$, this gives simpler expressions for the elementary symmetric polynomials than Theorem~\ref{mainthm}.

\begin{algorithm}\label{alg}\hfill \\
\textsc{Input:} Power polynomials $p_1,p_2,\ldots,p_\ell$.\\
\textsc{Output:} Expressions for $e_1,e_2,\ldots,e_n$ as rational functions of power polynomials.
\begin{enumerate} \item Recursively compute $e_1,\ldots,e_{r-1}$ via the Newton identities~\eqref{newt_eq}.
\item Let $\bar{e} = (e_n,e_{n-1},\ldots,e_1,1)^t$. Consider the Newton identities 
\begin{equation}\label{newtfinal}
ke_k - p_1e_{k-1} + p_2e_{k-2} - \ldots + (-1)^kp_k =0
\end{equation}
for all $k$ such that $r < k \leq \ell$ and $(k,r) = 1$. Let $M$ be the $(n-r+1) \times (n+1)$ matrix whose rows consist of the coefficients of~\eqref{newtfinal} for these values of $k$, so that $M\bar{e} = 0$.
\item Diagonalise the leftmost $(n-r+1) \times (n-r+1)$-block of $M$ to express $e_r,\ldots,e_n$ in terms of power polynomials.
\end{enumerate}
\end{algorithm}

\begin{example}Consider the case $r=3, n=7$. Since $\ell = 10$, we need to use the Newton identities for $k \in \{4,5,7,8,10\}$. We obtain the matrix equation
    \[
    \begin{pmatrix}
        0 & 0 & 0 & 1 & -p_1 & p_{2} & -p_3 & p_4 \\
        0 & 0 & 2 & -p_1 & p_2 & -p_3 & p_4 & -p_5 \\
        1 & -p_1 & p_2 & -p_3 & p_4 & -p_5 & p_6 & -p_7 \\
        -p_1 & p_2 & -p_3 & p_4 & -p_5 & p_6 & -p_7 & p_8 \\
        -p_3 & p_4 & -p_5 & p_6 & -p_7 & p_8 & -p_9 & p_{10}
    \end{pmatrix}
    \begin{pmatrix}
        e_7 \\ e_6 \\ e_5 \\e_4 \\ e_3 \\ e_2 \\ e_1 \\ 1
    \end{pmatrix}
    =
    \begin{pmatrix}
 0 \\ 0 \\ 0 \\ 0 \\ 0
    \end{pmatrix}.
    \]
    The leftmost $5 \times 5$-block in the above $5 \times 8$ matrix is invertible. Reducing it to the identity matrix via elementary row operations expresses each $e_3,\ldots,e_7$ as rational functions of power polynomials and $e_2$, which can be obtained from the usual Newton identity $2e_2 = e_1p_1 - p_2$.
\end{example}

\begin{example}
    For $r=3$ and $n = 4$, we have $2n+1-r = 6$ and $\ell = n + \lfloor (n-1)/(r-1) \rfloor = 5$. Algorithm~\ref{alg} applied in this setting gives the expression
    \[
    e_3(x_1,x_2,x_3,x_4) = \frac{p_1^2p_3 - p_1p_2e_2 + p_1p_4 +p_3e_2 + p_5}{p_2 - p_1^2},
    \]
    which is simpler than the expression \eqref{eqmod3} obtained from Algorithm~\ref{algR}.
\end{example}

\begin{remark}
The system of equations in Algorithm~\ref{alg} was also considered by Sch\"onhage, who commented \cite[p.~412]{schonhage}: ``\textit{...it is hard to see how this redundant system should be solved or how to prove its solvability.}" He then proves that $R(e_1,\ldots,e_n) = R(p_1,\ldots,p_\ell)$ in a different way. A posteriori, it follows that the system in question was solvable, but not necessarily in the way described by Algorithm~\ref{alg}. We sketch a proof of the algorithm here. We may and will assume without loss of generality that $R = \bF_r$.\\

Denote by $M(n)$ the leftmost $(n-r+1) \times (n-r+1)$-block of the matrix~$M$ from step~3 of Algorithm~\ref{alg}; if $n < r$, define $M(n)$ to be the empty matrix. Note that $M(r) = (-p_1)$, that 
\[
M(n) = 
\begin{pmatrix}
    0 & \ldots & 0 & 1 & -p_1 \\
    \vdots & \iddots & 2 & -p_1 & p_2 \\ 
    0 & \iddots & \iddots & & \vdots \\
    n-r & -p_1 & p_2 & \ldots & (-1)^{n+1} p_{n-r} \\
    -p_1 & p_2 & -p_3 & \ldots & (-1)^{n}p_{n-r+1}
\end{pmatrix}
\]
for $r+1 \leq n \leq 2r-2$, and that for $n \geq 2r-1$, we have the block decomposition
\[
M(n) = 
\left(
\begin{array}{@{}ccc|cccc|c@{}}
 &  &  & 0 & \ldots & 0 & 1 & -p_1 \\
 & 0 & &   \vdots & \iddots & 2 & -p_1 & p_2 \\ 
 &  & &   0 & \iddots & \iddots & \vdots & \vdots \\
 & & & r-1 & -p_1 & \ldots & -p_{r-2} & p_{r-1} \\ \hline
 & M(n-r) &  & & & * & & * \\ \hline
 & * & & & & * & & (-1)^{\ell - r} p_{\ell - r}
\end{array}
\right),
\]
where any entry in the regions denoted by~$*$ lies in~$R[p_1,\ldots,p_{\ell-r-1}]$. Note that $r \nmid \ell$.\\

We show that $\det M(n) \neq 0$ by induction on~$n$. If $r \leq n \leq 2r-2$, we see by the last row expansion of the determinant that
\[
\det M(n) = D \pm (n-r)! p_{n-r+1},
\]
for some $D \in R[p_1,\ldots,p_{n-r}]$. Since $r \nmid n+1$, Proposition~\ref{chain} implies that $\det M(n) \neq 0$. If $n \geq 2r-1$, we see in the same way that $\det M(n) = D \pm (r-1)!p_{\ell-r} \det M(n-r)$, 
with $D \in R[p_1,\ldots,p_{\ell-r-1}]$ and $\det M(n-r) \neq 0$ by induction. If $\det M(n) = 0$, this yields that $p_{\ell - r} \in R(p_1,\ldots,p_{\ell-r-1})$, but it is easily shown from first principles that $p_{\ell - r}$ is algebraically independent of $\{p_1,\ldots,p_{\ell-r-1}\}$ \cite[Proposition 1]{schonhage}.
\end{remark}


\subsection{Evaluating the expressions}\label{subsec:eval}
Suppose one wants to evaluate the expressions for the elementary symmetric polynomials obtained from Algorithm~\ref{algR} or \ref{alg} at an element $(\alpha_1,\ldots,\alpha_n) \in R^n$. Since the expressions are rational functions, this is not always possible. The problem (P) we are concerned with can be formulated in three equivalent ways:

\begin{enumerate}
    \item[(P1)] Given $(p_i(\alpha_1,\ldots,\alpha_n))_{i \geq 1}$, determine $e_k(\alpha_1,\ldots,\alpha_n)$ for $k = 1,\ldots,n$.
    \item[(P2)] Given a polynomial $f \in R[X]$ of degree~$n$ with roots $\alpha_1,\ldots,\alpha_n \in R$, determine the coefficients of~$f$ from the data $(p_i(\alpha_1,\ldots,\alpha_n))_{i \geq 1}$. 
    \item[(P3)] Let $M$ be a free $R$-module of rank~$n$ and let $T \in \End_R(M)$ be a linear operator with $n$ eigenvalues in~$R$. Given $(\Tr(T^i \hspace{0.1em} | \hspace{0.1em} M))_{i \geq 1}$, determine the characteristic polynomial of~$T$.
\end{enumerate}

To see why (P1) is equivalent to~(P2), note that the coefficients of a polynomial are (up to sign) the elementary symmetric polynomials in the roots. To see that (P2) is equivalent to~(P3), note that the eigenvalues of~$T$ are precisely the roots of the characteristic polynomial, and the trace of~$T^i$ is the $i$-th power polynomial in the eigenvalues of~$T$.\\

By the example given in the introduction, it is necessary to have some conditions on $\alpha_1,\ldots,\alpha_n$ in order to guarantee a solution to~(P). The following result from the literature states that if~$R$ is a field, the conditions are as good as one can hope for.

\begin{proposition}\label{prop:recover}
    Let $K$ be a field of characteristic $r > 0$, and let $\alpha_1,\ldots,\alpha_n \in K$. Then one can determine $e_k(\alpha_1,\ldots,\alpha_n)$ for $k = 1,\ldots,n$ from $(p_i(\alpha_1,\ldots,\alpha_n))_{i=1}^{2n}$ if and only if each~$\alpha_j$ occurs with multiplicity less than~$r$.
\end{proposition}

\begin{proof}
    If some $\alpha_j$ is repeated $r$ times, the contribution to $p_i(\alpha_1,\ldots,\alpha_n)$ is $r \alpha_j^i = 0$ for all~$i$. Hence these elements cannot be detected by the power polynomials. For the converse, see~\cite[Prop.~2]{bostan-FSS}.
\end{proof}

\begin{remark}
    One can also use Proposition~\ref{prop:recover} to give a sufficient condition for~(P) to be solvable over reduced rings~$R$, as these can be embedded into a product of fields. For general rings~$R$, Remark~\ref{remk} gives a sufficient condition, but due to the possibility of removable poles this condition is not always necessary.
\end{remark}

\section{Discrete valuation rings}\label{sec:dvr}
In this section, we consider a question about discrete valuation rings in positive characteristic. Let $(K,v)$ be a discretely valued field of characteristic~$r$ such that $v$ is trivial on the prime subfield of~$K$. Given an irreducible polynomial $P \in K[X]$, the valuations of the coefficients of~$P$ are related via the Newton polygon to the valuations of its roots $\alpha_1,\ldots,\alpha_n$. In characteristic zero, combining this with Newton's identities implies that $\min v(\alpha_i) \geq m$ if and only if $v(p_k(\alpha_1,\ldots,\alpha_n)) \geq km$ for $k = 1,\ldots,n$. In positive characteristic, this is no longer true. However, we can salvage the result if we have information about the valuations of \emph{all} power polynomials. We start with a lemma.

\begin{lemma}\label{lem:sums}
    Let $R$ be an integral domain of characteristic~$r$. Let $x_1,\ldots,x_n \in R$ such that there exists some $x_i \neq 0$ in the list which is not repeated a multiple of $r$ times. Then there exists $s \geq 1$ such that $x_1^s + \ldots + x_n^s \neq 0$.
\end{lemma}

\begin{proof}
    Write $S = \{x_1,\ldots,x_n \} \setminus \{0\}$ and let $d := |S| \geq 1$. Let $y_1,\ldots,y_d \in S$ be distinct. For $i = 1,\ldots,d$, write
    \[
    n_i := |\{1 \leq j \leq n \ | \ x_j = y_i \}|.
    \]
    By assumption, some $n_i$ is non-zero in~$R$. We have the following equation:
    \[
    \begin{pmatrix}
        y_1 & y_2 & \cdots & y_d \\
        y_1^2 & y_2^2 & \cdots & y_d^2 \\
        \vdots & \vdots & \ddots & \vdots \\
        y_1^d & y_2^d & \cdots & y_d^d
    \end{pmatrix}
    \begin{pmatrix}
        n_1 \\ n_2 \\ \vdots \\ n_d
    \end{pmatrix}
    =
    \begin{pmatrix}
        p_1(x_1,\ldots,x_n) \\
        p_2(x_1,\ldots,x_n) \\
        \vdots \\
        p_d(x_1,\ldots,x_n)
    \end{pmatrix}.
    \]
    The matrix on the left is invertible over $\text{Frac}(R)$, as it is of Vandermonde type with determinant 
    \[
    e_d(y_1,\ldots,y_n) \prod_{1 \leq i < j \leq d}(y_i - y_j) \neq 0.
    \]
    Hence there is some $1 \leq s \leq d$ such that $p_s(x_1,\ldots,x_n) \neq 0$.
\end{proof}

\begin{proposition}
    Let $(K,v)$ be a discretely valued field of characteristic~$r$ with residue field $\bF_q$. Let $\alpha_1,\ldots,\alpha_n \in K$ and fix $m \in \mathbb Z$. Suppose that some $\alpha_j$ with minimal valuation (i.e., satisfying $v(\alpha_j) = \min v(\alpha_i)$) occurs with multiplicity not divisible by~$r$. Then the following are equivalent:
    \begin{enumerate}
        \item[$(1)$] $\min v(\alpha_i) \geq m$;
        \item[$(2)$] $v(p_k(\alpha_1,\ldots,\alpha_n)) \geq km \ \forall k \geq 1$.
    \end{enumerate}
\end{proposition}

\begin{proof}
    $(1) \implies (2)$ follows directly from the triangle inequality. Conversely, by the previous discussion we may assume $r > 0$. By replacing $K$ with its completion $K_v$, we may assume that $K = \bF_q(\!(\pi)\!)$ for some uniformizer $\pi$ \cite[Thm II.2]{serre_local}. Write $m' := \min v(\alpha_i)$ and assume without loss of generality that
    \[
    v(\alpha_1) = v(\alpha_2) = \ldots = v(\alpha_d) = m' < \infty, \qquad v(\alpha_j) \geq m'+1 \ \text{for} \ j=d+1,\ldots,n.
    \]
    Since $\alpha_1,\ldots,\alpha_d$ satisfy the assumption of Lemma~\ref{lem:sums}, we can find $s \geq 1$ such that 
    \[
    y := \sum_{i=1}^d \alpha_i^s \neq 0.
    \]
    Write $M := v(y) \geq m's$ and $y = \sum_{i \geq M} y_i \pi^i$, so that $y_M \in \bF_q^\times$.\\
    
    For $i = 1,\ldots,d$, write $\alpha_i = \pi^{m'} u_i$ for some $\pi$-adic units $u_i \in \bF_q^\times + \pi\bF_q[\![\pi]\!]$. In particular, $u_i^{(q-1)q^N} \in 1 + \pi^{q^N} \bF_q[\![\pi]\!]$. Hence we have for $q^N > M$,
    \begin{align*}
    p_{(q-1)q^N+s}(\alpha_1,\ldots,\alpha_d) &= \sum_{i=1}^d \pi^{m'(q-1)q^N}u_i^{(q-1)q^N}\alpha_i^{s} \\
    &= \pi^{m'(q-1)q^N}\sum_{i=1}^d \alpha_i^s + \text{h.o.t.} \\
    &= y_M \pi^{m'(q-1)q^N + M} + \text{h.o.t.}
    \end{align*}
    and hence $v(p_{(q-1)q^N + s}(\alpha_1,\ldots,\alpha_d)) = m'(q-1)q^N +M$. In particular, for $N \gg 1$, we have
    \[
    v(p_{(q-1)q^N+s}(\alpha_{d+1},\ldots,\alpha_{n})) \geq (m'+1)((q-1)q^N+s) > v(p_{(q-1)q^N+s}(\alpha_1,\ldots,\alpha_d)),
    \]
    and hence
    \[
    v(p_{(q-1)q^N+s}(\alpha_1,\ldots,\alpha_n)) = m'(q-1)q^N+M.
    \]
    By assumption, we have $v(p_{(q-1)q^N + s}(\alpha_1,\ldots,\alpha_n)) \geq ((q-1)q^N + s)m$. Hence the result follows because
    \[
    \lim_{N \to \infty} \frac{m'(q-1)q^N +M}{(q-1)q^N + s} = m'.
    \]
\end{proof}

\subsection*{Acknowledgements}
I am grateful to Per Alexandersson for helpful discussions and providing a combinatorist's perspective on the contents of this paper. I thank Bruce Sagan and Chris Bowman for helpful comments about the existing literature, and Alin Bostan for his comments and for pointing out the papers \cite{schonhage,kakeya}. I would also like to thank my PhD-advisor Jonas Bergstr\"om and my co-advisor Olof Bergvall. Lastly, I thank the anonymous referee for their careful reading.

\printbibliography

\textsc{Matematiska institutionen, Stockholms universitet, 106 91 Stockholm, Sweden}\\
\textit{E-mail address:} \href{mailto:sjoerd.devries@math.su.se}{\texttt{sjoerd.devries@math.su.se}}
\end{document}